\documentclass[12pt]{compositio}
\usepackage{amsmath,mathtools}

\theoremstyle{definition}
\newtheorem{theorem}[equation]{Theorem}
\newtheorem{lemma}[equation]{Lemma}
\newtheorem{proposition}[equation]{Proposition}
\newtheorem{corollary}[equation]{Corollary}

\newtheorem{definition}[equation]{Definition}

\newtheorem{example}[equation]{Example}
\newtheorem{question}[equation]{Question}
\newtheorem{conjecture}[equation]{Conjecture}
\newtheorem{remark}[equation]{Remark}

\numberwithin{equation}{section}

\newcommand{\bC}{\mathbf{C}}

\newcommand{\bA}{\mathbf{A}}
\newcommand{\Spec}{\operatorname{\mathsf{Spec}}}

\newcommand{\cO}{\mathcal{O}}
\newcommand{\bZ}{\mathbf{Z}}

\newcommand{\rmD}{\mathrm{D}}
\newcommand{\caD}{\mathcal{D}}

\newcommand{\Hom}{\operatorname{Hom}}

\newcommand{\Id}{\operatorname{Id}}

\newcommand{\gr}{\operatorname{gr}}

\newcommand{\IC}{\operatorname{IC}}
\newcommand{\Ext}{\operatorname{Ext}}
\newcommand{\coker}{\operatorname{coker}}

\newcommand{\db}{\operatorname{db}}

\AtBeginDocument{%
   \def\MR#1{}
}

\title{On $\caD$-modules related to the $b$-function and Hamiltonian flow}
\author{Thomas Bitoun}
\email{tbitoun@gmail.com}
\address{Mathematical Institute\\University of Oxford\\Oxford OX2 6GG\\UK}

\author{Travis Schedler}
\email{trasched@gmail.com}
\address{Department of Mathematics\\Imperial College London\\London SW7 2AZ\\UK}

\thanks{
The first author was supported by EPSRC grant EP/L005190/1.
The second author was partly supported by US National Science
Foundation Grant DMS-1406553, and is grateful to the Max Planck Institute for Mathematics in Bonn for excellent working conditions.}

\begin{document}


\begin{abstract}
  Let $f$ be a quasi-homogeneous polynomial with an isolated
  singularity in $\bC^n$.  We compute the length of the $\caD$-modules $\caD
  f^\lambda/\caD f^{\lambda+1}$ 
  generated by complex powers of $f$ in terms of the Hodge filtration
  on the top cohomology of the Milnor fiber.  When $\lambda = -1$ we
  obtain
  one more than the reduced genus of the singularity
  ($\dim H^{n-2}(Z,\cO_Z)$ for $Z$ the exceptional fiber of a
  resolution of singularities).
  We conjecture that this holds without the quasi-homogeneous
  assumption.  We also deduce that the quotient
  $\caD f^\lambda / \caD f^{\lambda+1}$ is nonzero when $\lambda$ is a
  root of the $b$-function of $f$ (which Saito recently showed fails
  to hold in the inhomogeneous case).  We obtain these results by
  comparing these $\caD$-modules to those defined by Etingof and the
  second author which represent invariants under Hamiltonian flow.
\end{abstract}

\maketitle


\section{Introduction}

\subsection{Length of $\caD  \frac{1}{f}$ and local cohomology}

Throughout the paper, $R$ will denote the ring of complex polynomials in $n \geq 3$ variables and $\mathcal{D}$ the
ring of complex polynomial differential operators in $n$ variables.
For $f \in R$, we consider the left $\caD$-modules $R$ and $R[\frac{1}{f}]$.

In this paper we compute the length of the left
$\mathcal{D}$-submodule of $R[\frac{1}{f}]$ generated by $\frac{1}{f}$
for $f$ a quasi-homogeneous polynomial with an isolated
singularity. We denote this module by $\caD \frac{1}{f}$. Actually it
turns out to be more natural to work with the quotient
$R[\frac{1}{f}]/R$, which is the first local cohomology group
$H^1_f(R)$, and its left $\mathcal{D}$-submodule generated by the class
of $\frac{1}{f}$, i.e. $\caD \frac{1}{f} / R$. As $R$ is simple, its length is precisely one less.

Let us state our main result, Theorem \ref{t:mt-qh}, whose proof is in Section \ref{s:pf-c:mc}, in the case of $f$ a homogeneous
polynomial with an isolated singularity, i.e., the zero-locus of $f$
is the cone over a smooth projective variety $Y$ of dimension $n-2$. 

\begin{theorem} \label{t:mt-homogen} 
The length of the left
  $\mathcal{D}$-submodule of $H^1_f(R)$ generated by the class of
  $\frac{1}{f}$ is $1+\dim H^{n-2}(Y, \mathcal{O}_Y).$
\end{theorem}  

\begin{remark} As an application, we can obtain a necessary and
  sufficient condition for the left $\mathcal{D}$-module
  $R[\frac{1}{f}]$ to be generated by $\frac{1}{f}$.  It is well known
  that $R[\frac{1}{f}]$ is generated by $\frac{1}{f}$ if and only if
  the Bernstein-Sato polynomial of $f$ has no other integral root than
  $-1.$ Indeed for $\lambda$ the lowest integral root,
  $\mathcal{D} f^\lambda / \mathcal{D} f^{\lambda+1} \neq 0$ by
  \cite[Proposition 6.2]{MR0430304}. 
On the other hand, in the case of
  the theorem the length of $R[\frac{1}{f}]/R = H^1_f(R)$ is
  $1+\dim H^{n-2}(X^\circ, \bC)$, where
  $X^\circ = \{f=0\} \setminus\{0\}$ is the punctured cone over $Y$,
  as follows from Lemmas \ref{l:j*maxind} and \ref{l:j!-sub}, via
  Verdier duality. Note that $H^{n-2}(X^\circ, \bC)$ identifies with
  the space of primitive cohomology classes of $Y$ in degree
  $n-2$. Therefore the theorem implies that $H^1_f(R)$, and hence
  $R[\frac{1}{f}]$, is generated by $\frac{1}{f}$ if and only if the
  space of primitive cohomology classes of $Y$ in degree $n-2$ equals
  $H^{n-2}(Y,\mathcal{O}_Y)$ (by the Hodge decomposition, the latter
  is always a summand of the former, so the statement is that this
  summand is everything).
  \end{remark}
  
\begin{example} \label{e:curve} When $n=3$, $Y$ is a smooth curve of genus
  $g= \dim H^1(Y, \mathcal{O}_Y).$ In this case $\dim
  H^1(X^\circ, \bC) = \dim H^1(Y,\bC) = 2g$ so that the length of
  $H^1_f(R)$ is $1+2g$. On the other hand, by the theorem the length of
  the submodule generated by $\frac{1}{f}$ is only $1+g$.  
\end{example}

In the case of a general $f$ with an isolated singularity, the genus
generalizes as follows.
\begin{definition} \label{definition: reduced genus} Let $X$ be a
  variety with an isolated singularity at $x \in X$. Let
  $\rho: \widetilde X \to X$ be a resolution of singularities and
  $Y := \rho^{-1}(x)$.
  The reduced genus $g_x$ of the singularity $x$ is defined as
  $\dim H^{n-2}(Y, \mathcal{O}_Y).$
\end{definition}
The reduced genus is independent of the choice of resolution since the
derived pushforward $R\rho_* \mathcal{O}_{\widetilde X}$ is
independent of the choice of resolution (any two resolutions can be
dominated by a third one, and for a resolution of a smooth variety
this pushforward is underived). Then, one can show that
$H^{n-2}(Y,\mathcal{O}_Y)$ is
the fiber at $x \in X = \Spec R^0 \rho_* \mathcal{O}_{\widetilde X}$
of $R^{n-2} \rho_* \mathcal{O}_{\widetilde X}$ (since $H^{>n-2}(Y,\mathcal{O}_Y)=0$).
\begin{remark} If $f$ is homogeneous with an isolated singularity at
  the origin, then $g_0$ is equal to $\dim H^{n-2}(Y', \mathcal{O}_{Y'})$, where
  the zero-locus of $f$ is the cone over a smooth projective variety
  $Y'$ of dimension $n-2$. Indeed $Y'$ is the exceptional fiber of the
  blow-up of the origin.\end{remark}
\begin{remark} We call the dimension of the fiber
  $H^{n-2}(Y, \mathcal{O}_Y)$ of
  $R^{n-2} \rho_* \mathcal{O}_{\widetilde X}$ at $x$ the reduced genus,
  in contrast to the geometric genus, which is the dimension of the
  entire stalk of $R^{n-2} \rho_* \mathcal{O}_{\widetilde X}$ at $x$,
  i.e., $\dim H^{n-2}(\rho^{-1}(U), \mathcal{O}_{\rho^{-1}(U)})$ for
  $U$ a small ball around $x$.  Already when $X$ is the cone over a
  smooth projective plane curve $Y$ of degree $d \geq 4$, the
  geometric genus exceeds the reduced genus, with the former
  $\frac{d(d-1)(d-2)}{6}$ and the latter $\frac{(d-1)(d-2)}{2}$ (the
  genus of $Y$).
\end{remark}



Our theorem is a particular case of the following conjecture: 
\begin{conjecture}\label{c:mc} Let $f$ be an irreducible
complex polynomial with a unique
  singularity $x$. The length of the left $\mathcal{D}$-submodule of
  $H^1_f(R)$ generated by $\frac{1}{f}$ is $1 + g_x.$
\end{conjecture}





\subsubsection*{Supporting evidence for Conjecture \ref{c:mc}}

We now would like to list some supporting evidence for the conjecture.
\begin{enumerate}
\item 
We prove the conjecture in the case $f$ is quasi-homogeneous, which is
one of the main results of this article (Theorem \ref{t:mt-qh}). This generalizes
Theorem \ref{t:mt-homogen}.


\item 
There is the following analogous result in positive characteristic, which 
motivated the project: 

Suppose that $f$ is a polynomial in $n$ variables with rational
coefficients and an isolated singularity at the point $0.$ 
Let
$\widetilde X \xrightarrow{\rho} X_0$ be a resolution of the local
singularity, of exceptional fiber $Y$, as above. Consider further a
prime number $p$ large enough so that the denominators of the
coefficients of $f$ are prime to $p$ and the reduction modulo $p$ of
$\rho$ 
is a well-defined resolution
$\widetilde X_p \xrightarrow{\rho_p} (X_0)_p$ of the local singularity
$(X_0)_p.$ Denote by
$\widetilde X_{\overline{p}} \xrightarrow{\rho_{\overline{p}}}
(X_0)_{\overline{p}}$
the base-change of $\rho_p$ to an algebraic closure $k$ of
$\mathbb{F}_p$ and let $Y_{\overline{p}}$ be its exceptional
fiber. 
Then $H^{n-2}(Y_{\overline{p}},\mathcal{O}_{Y_{\overline{p}}})$ has a
natural Frobenius action and hence is a left $k[F]$-module. Here is
the relevant
genus: 

\begin{definition} 
\begin{itemize}

\item Let $V$ be a $k[F]$-module. The stable part $V_s$ of $V$ is $\cap_{r\geq1}F^r(V).$
\item The $p$-genus $g'_p$ is the dimension of the stable part of $H^{n-2}(Y_{\overline{p}},\mathcal{O}_{Y_{\overline{p}}}).$
\end{itemize} 
\end{definition}

Let $R_p$ be the ring of $k$-polynomials in $n$ variables, $\caD_p$ be
the ring of Grothendieck differential operators on the affine space
$\mathbb{A}^n_k$ 
and $f_p$ be the reduction of $f$ modulo $p.$ We have the following
result of the first author, see \cite{doi:10.1093/imrn/rny058}:

\begin{theorem} \label{theorem: char p} Let $f$ be an absolutely
  irreducible polynomial in $n\geq3$ variables with rational
  coefficients. Suppose that $f$ has a unique singularity.
  Then for all primes $p$ large enough, the length of the
  left $\caD_p$-module $H^1_{f_p}(R_p)$ is $1+ g'_p.$\end{theorem}


\begin{remark} By \cite[Theorem 1.1]{MR2155224} the left
  $\caD_p$-module $H^1_{f_p}(R_p)$ is generated by the class of
  $\frac{1}{f_p}.$
  So in fact by Theorem \ref{theorem: char p}, for $p$ large enough
  the length of the left $\caD_p$-submodule generated by the class of
  $\frac{1}{f_p}$ is $1+ g'_p.$\end{remark}
\begin{remark} It is expected, at least in the smooth and connected
  case, that for infinitely many primes $F$ acts bijectively on
  $H^{n-2}(Y_{\overline{p}},\mathcal{O}_{Y_{\overline{p}}})$, see
  \cite[Conjecture 1.1]{MR2863367}. Thus for infinitely many primes
  one would have $g'_p=g_0.$ Hence for infinitely many primes the
  length of $H^1_{f_p}(R_p)$ would equal the length of the
  $\mathcal{D}$-submodule of $H^1_f(R)$ generated by $\frac{1}{f}.$
\end{remark}

\item 
  Another piece of evidence for our conjecture comes from
  certain $\mathcal{D}$-modules which control invariants under the
  Hamiltonian flow, see Subsection \ref{subsection: HamFlow}. We show
  that the corresponding left $\caD$-modules surject to
  $\mathcal{D} \frac{1}{f}/R.$ As a consequence we prove that the above
  conjecture is equivalent to \cite[Conjecture 3.8]{MR3283930}.  In
  fact this is how we prove the conjecture in the quasi-homogeneous
  case, since by \cite[Proposition 3.11]{MR3283930} the aforementioned
  conjecture holds in this case.

\item
  As an additional piece of evidence for the conjecture, we briefly
  consider a simple example which is not quasi-homogeneous:

  \begin{example} 
    Let $X$ be one of the $T_{p,q,r}$ singularities at the origin in
    $\bA^3$, namely $f=x^p + y^q + z^r + xyz$ where
    $p^{-1}+q^{-1}+r^{-1} < 1$.  By \cite[Example 4.13]{MR499664}, the
    only integral root of the $b$-function in this case is $-1$. This
    implies that $H^1_f(R)$ is generated by $\frac{1}{f}$ as a
    $\caD$-module.
    So here the conjecture reduces to: using the notation of
    Definition \ref{definition: reduced genus},
    the length of $H^1_f(R)$ is one more than
    $\dim H^1(Y, \mathcal{O}_Y)$. For a general isolated singularity in $\bA^n$,
    the length of the former is $1+\dim H^{n-2}(V\setminus \{0\},\bC)$
     (by Lemmas \ref{l:j*maxind} and the Verdier dual statement of \ref{l:j!-sub}), where $V$ is a contractible neighborhood of the singularity
    $\{0\}$,
so in our case we reduce to showing that
    $\dim H^{1}(V\setminus \{0\}, \bC) = \dim
    H^{1}(Y,\mathcal{O}_Y)$.
    For a general isolated singularity in $\bA^3$,
    $\dim H^{1}(V\setminus \{0\}, \bC) = 2g+b$, by
    \cite[II.(3.4)]{Dimca}, where $g$ is the sum of the genera of the
    components of $\widetilde X$ and $b$ is the first Betti number of
    the dual graph of $\widetilde X$. On the other hand,
    $\dim H^{1}(Y,\mathcal{O}_Y) = g+b$.  So the two are equal
    exactly when $g=0$.  
This is the case for 
    $T_{p,q,r}$ by 
    \cite[II.(4.5)]{Dimca}. 
    Hence we conclude that the conjecture holds for the $T_{p,q,r}$
    singularities.  

    Note that, by the above, a special case of the conjecture is the
    following statement: When $n=3$ and $f$ has an isolated
    singularity, then $-1$ is the only integral root of the
    $b$-function (i.e., $H^1_f(R)$ is generated over $\caD$ by
    $\frac{1}{f}$) if and only if the genera of the components of a
    resolution $\widetilde X$ of $X=\{f=0\}$ are all zero.
\end{example} 

\end{enumerate}

\begin{remark} We can relax the
  assumptions of the conjecture. Namely, assume only that $f$ is a
  polynomial with isolated singularities (not necessarily
  irreducible). Since $n \geq 3$, the hypersurface $X=\{f=0\}$ is
  normal, so its irreducible components equal its connected
  components.  By localization, the conjecture easily
  implies that the length of the left $\caD$-submodule of $H^1_f(R)$
  generated by $\frac{1}{f}$ is then equal to the number of components
  of $X$ plus the sum of the reduced genera of the isolated
  singularities.
\end{remark}
\begin{remark} The conjecture as well as the results of this article
  also extend to the case $n=2$, provided we replace the reduced genus
  by $\dim H^0(Y, \cO_Y)/\bC$ and the cohomology $H^0(X^\circ, \bC)$
  by the reduced cohomology $H^0(X^\circ,\bC)/\bC$.
\end{remark}


\subsection{$\caD$-modules related to the $b$-function} \label{ss:dmrbf}
Next, we consider our problem for the variant of
$\mathcal{D}\frac{1}{f}$ where $\frac{1}{f}$ is replaced by
$f^\lambda$ for an arbitrary complex number $\lambda$. We are
interested in the length of the $\mathcal{D}$-module
$\mathcal{D}f^{\lambda}/\mathcal{D}f^{\lambda+1}$.

Let us recall the definition. First of all, the variant of the
$\mathcal{D}[\frac{1}{f}]$-module $R[\frac{1}{f}]$ is the
$\mathcal{D}[\frac{1}{f}]$-module $R[\frac{1}{f}]f^{\lambda}.$ That is
the trivial line bundle on $\Spec(R[\frac{1}{f}])$ with generator
$f^\lambda$, on which a derivation $\partial$ acts by
$\partial \cdot f^\lambda= \lambda \partial(f)f^{-1}f^\lambda.$

\begin{definition}
  The left $\mathcal{D}$-module $\mathcal{D}f^{\lambda}$ is the left
  $\mathcal{D}$-submodule of $R[\frac{1}{f}]f^{\lambda}$ generated by
  $f^{\lambda}$, and $\mathcal{D}f^{\lambda+1}$ is the left
  $\mathcal{D}$-submodule generated by
  $f \cdot f^{\lambda}$.
\end{definition}

This definition is reminiscent of that of the $b$-function of $f.$
Namely the $b$-function is the minimal polynomial of the action of $s$
on the left $\mathcal{D}[s]$-module
$\mathcal{D}[s]f^s/\mathcal{D}[s]f^{s+1}$, where $s$ is a 
variable commuting with $\mathcal{D}.$ More precisely, we consider
$\mathcal{D}[\frac{1}{f}][s]f^s$, the free $R[\frac{1}{f}][s]$-module
of rank one with generator $f^s$, on which derivations $\partial$ in
$\mathcal{D}$ act by $\partial \cdot f^s= s \partial(f)f^{-1}f^s.$

\begin{definition}The left $\mathcal{D}[s]$-modules
  $\mathcal{D}[s]f^s$ and $\mathcal{D}[s]f^{s+1}$ are the left
  $\mathcal{D}[s]$-submodules of $\mathcal{D}[\frac{1}{f}][s]f^s$
  generated by $f^s$ and $f\cdot f^s$, respectively.\end{definition}

From the definitions it is clear that there is a
surjection
$$\mathcal{D}[s]f^s/\mathcal{D}[s]f^{s+1} \xrightarrow{p_\lambda}
\mathcal{D}f^{\lambda}/\mathcal{D}f^{\lambda+1},$$
sending $s$ to $\lambda$ and $f^s$ to $f^\lambda.$ Thus the quotient
$\mathcal{D}f^{\lambda}/\mathcal{D}f^{\lambda+1}$ certainly vanishes
if $\lambda$ is not a root of the $b$-function of $f.$ One may wonder
about the converse, namely:
\begin{question}\label{q:nontriv}
  Suppose that $\lambda$ is a root of the $b$-function of $f.$ Is the
  quotient $\mathcal{D}f^{\lambda}/\mathcal{D}f^{\lambda+1}$
  nontrivial?
\end{question}
In this paper we compute the length of
$\mathcal{D}f^{\lambda}/\mathcal{D}f^{\lambda+1}$ when $f$ is
quasi-homogeneous with an isolated singularity.  As a consequence we
answer Question \ref{q:nontriv} affirmatively in this
case. 
\begin{remark}
  Note, on the other hand, that Saito shows in \cite[Example 4.2]{Sai-dmgrphf}
  that for all sufficiently small $\epsilon>0$ and for
  $f = x^{14}+y^{14}-x^6y^6 +z^5 -\epsilon x^9 y^2z,
  \mathcal{D}f^{\lambda}/\mathcal{D}f^{\lambda+1}$
  vanishes at $\lambda=-\frac{13}{30}$ even though $-\frac{13}{30}$ is
  a root of the $b$-function of $f.$ Thus the answer to Question
  \ref{q:nontriv} is negative in general, even for an isolated
  singularity. 
\end{remark}

Let
$\caD[s] f^s / \caD[s] f^{s+1}|_{s=\lambda} := \caD[s] f^s / \caD[s]
f^{s+1} \otimes_{\bC[s]} \bC[s]/(s-\lambda)$
be the restriction to $s=\lambda.$ Our result is the
following: 
\begin{theorem} \label{t:lneq-1} Let $f$ be a complex polynomial in $n\geq3$ variables, $\lambda\in \mathbb{C}$ and let  
\[
\caD[s] f^s / \caD[s] f^{s+1}|_{s=\lambda} 
 \xrightarrow{q_{\lambda}} \mathcal{D}f^{\lambda}/\mathcal{D}f^{\lambda+1}
\]
be the canonical surjective morphism. 
Suppose that $f$ is quasi-homogeneous with an isolated singularity at $0.$
\begin{enumerate}
\item \label{item: lambda} If $\lambda \neq -1$, then the canonical
  morphism $q_{\lambda}$ is an isomorphism.
\item \label{item: -1} If $\lambda=-1$, then the kernel of
  $q_{\lambda}$ is isomorphic to
  $\delta_0^{\dim H^{n-2}(X^\circ, \mathbb{C})- g_0}$, where
  $\delta_0$ is the irreducible $\caD$-module supported at $0$ and
  $X^\circ$ is the zero-locus of $f$ minus its
  singularity. \end{enumerate}
\end{theorem}
For $\lambda \neq -1$, the length of $\caD[s] f^s / \caD[s]
f^{s+1}|_{s=\lambda}$ was computed in
\cite[\S 6.4]{MR1943036}.  
To express the answer we need some
notation.
Since $f$ is quasi-homogeneous, we may and do endow the polynomial algebra $R=\mathbb{C}[x_1, \dots, x_n]$ with a nonnegative grading such that $x_i$ is homogeneous of degree $m_i\geq1$ for all $i$, and $f$ is a homogeneous element of $R.$ We denote the degree of a homogeneous element $g$ of $R$ by $|g|.$
Let $J :=R / (\frac{\partial f}{\partial x_1},
\ldots, \frac{\partial f}{\partial x_n})$ be the Jacobi ring, which has finite complex dimension
equal to the Milnor number $\mu$ of the isolated singularity. The grading on $R$ induces a grading on the Jacobi ring and the above length is equal to the dimension of a certain homogeneous component. 
Using the isomorphism of Theorem \ref{t:lneq-1}, we have:
\begin{corollary} \label{c:lneq-1} If $\lambda \neq -1$, then the
  length of $\mathcal{D}f^{\lambda}/\mathcal{D}f^{\lambda+1}$ is
  $\dim_{\mathbb{C}} J_{-|f|\lambda - \sum_i m_i}$.
\end{corollary}
\begin{remark}
  Note that, since $f$ is quasi-homogeneous, $s$ acts semisimply on
  $(s+1)\caD[s] f^s / \caD[s] f^{s+1}$, so in the case $\lambda \neq -1$, the restriction $\caD[s] f^s / \caD[s] f^{s+1}|_{s=\lambda}$ in the theorem can be identified with the generalized eigenspace of  $\caD[s] f^s / \caD[s] f^{s+1}$
under $s$ of eigenvalue $\lambda$. This is not true for $\lambda=-1$, since $s$ does not act semisimply on $\caD[s] f^s / \caD[s] f^{s+1}$ itself in general. The advantage of using generalized eigenspaces is that it gives a direct sum decomposition of $\caD[s] f^s / \caD[s] f^{s+1}$. See Remark \ref{r:ge-ds} for 
more details on the summands.
\end{remark}

\begin{remark} 
  The dimensions in Corollary \ref{c:lneq-1} are of a Hodge-theoretic
  nature, see Theorem \ref{t:steenbrink}
  below. 
\end{remark}

\begin{remark}
  It would be nice to generalize Conjecture \ref{c:mc} to the length
  of $\caD f^{\lambda} / \caD f^{\lambda+1}$, for at least some roots
  $\lambda$ of the $b$-function. \end{remark}


\subsection{$\caD$-modules related to Hamiltonian
  flow} \label{subsection: HamFlow}

Our main technique of proof is to compare the various $\caD$-modules
above with certain modules attached to the Hamiltonian flow.



Let us first say a few words about D-modules on
singular varieties. Recall, following Kashiwara, that if $X$ is
a singular affine variety equipped with an embedding $X \to Y$ into
a smooth affine variety, one can define the category of ``D-modules on $X$'' to be the category of right $\mathcal{D}_{Y}$-modules set-theoretically supported on $X$.  This category does not depend on the choice of embedding up to canonical equivalence. It is \emph{not} in general equivalent to the (full) category of modules over any ring. In order to avoid conflict with our distinct notation $\mathcal{D} = \mathcal{D}(\bA^n)$ for the ring of polynomial differential operators in $n$ variables, we will use the roman font ``D-modules'' when referring to this category.


We consider the functor $\Gamma_{\mathrm{D}}(X,-)$ from the category of 
D-modules on $X$ 
to vector spaces, given by taking the global sections
scheme-theoretically supported on $X$. This functor does not depend on
the choice of embedding up to canonical isomorphism. Moreover, for
every D-module $M$ on $X$, there is a canonical action on
$\Gamma_{\rmD}(X,M)$ of the Lie algebra of global vector fields
on $X$, coming from the action on $M$ of vector fields on $Y$, which
again does not depend on the choice of embedding.




Let $\mathfrak{g}$ be a Lie algebra and suppose that $X$ is a singular
affine variety equipped with an action $\theta: \mathfrak{g} \to T_X$.
The D-module on $X$ we are interested in is the 
module 
$M(X,\theta)$ representing the functor of
$\mathfrak{g}$-invariant global sections. Namely, this D-module
is such that, for any
 D-module $N$ on $X$, we have a functorial isomorphism
$\Hom
(M(X,\theta), N)
\cong\Gamma_{\mathcal{D}}(X,N)^{\mathfrak{g}}.$
The D-module $M(X,\theta)$ can be described explicitly as follows. 

Let $X \to Y$ be a closed embedding of $X$ into a smooth affine
variety $Y$ with ideal $I_X.$ There is a standard D-module $D_X$ on
$X$, which corresponds to the right $\mathcal{D}_Y$-module
$I_X \caD_Y \setminus \caD_Y$ supported on $X$. This does not depend
on the choice of the embedding $X \to Y$ via Kashiwara's
equivalence. Moreover, \emph{left} multiplication in $\caD_Y$ induces
an action on $D_X$ of the vector fields on $X$ by D-module
endomorphisms; via $\theta$ we obtain an action of $\mathfrak{g}$.
Neither $D_X$ nor the above actions depend on the choice of embedding.
The D-module $M(X,\theta)$ on $X$ representing the invariants
under the Hamiltonian flow is then
$\mathfrak{g} \cdot D_X \setminus D_X$.

We now recall from \cite[\S 3.4]{ES-dmlv} the example of interest to us.
Let $f$ be a complex polynomial in $n\geq3$ variables and $X := \{f=0\} \subseteq \bA^n = Y.$ 
To every differential form $\alpha$ of degree 
$(n-3)$ on $\bA^n$, we associate the vector field
$\xi_\alpha := (\partial_1 \wedge \cdots \wedge \partial_n) (d \alpha
\wedge df)$.  It is parallel to $X$, i.e. $\xi_\alpha(f)=0$, and therefore defines a vector
field $\xi_\alpha|_X$ on $X.$ 
\begin{definition} The Lie algebra $\mathfrak{g}(f)$ is the vector space $\Omega^{n-3}(\bA^n)$ of $(n-3)$-forms on $\bA^n$, endowed with the Lie bracket $$\{\alpha,
\beta\} := L_{\xi_\alpha}(\beta) 
,$$ for all $\alpha, \beta \in \Omega^{n-3}(\bA^n)$, with
$L_\xi$ the Lie derivative. We define the Lie homomorphism $\theta: \mathfrak{g}(f) \to T_X$ 
by $\theta(\alpha)=\xi_\alpha|_{X}$, 
for all $\alpha\in \Omega^{n-3}(\bA^n).$\end{definition}

Note that we also have a Lie homomorphism $\widetilde \theta: \mathfrak{g}(f) \to T_{\bA^n}$ defined by $\widetilde
\theta(\alpha)=\xi_\alpha$, for all $(n-3)$-forms $\alpha$ on $\bA^n.$

We then form the D-module $M(X,\theta)$ on $X$.  Let us explicitly write the corresponding left $\caD$-module (on $\bA^n$).
\begin{definition} Let $i_*M(X,\theta)$ be the right $\caD$-module on $\bA^n$ corresponding to $M(X,\theta).$ We define $M(f)$ to be the left $\caD$-module on $\bA^n$ obtained by applying to $i_*M(X,\theta)$ the standard equivalence between right and left $\caD$-modules on $\bA^n.$
\end{definition}
It follows directly from the definitions
that
\begin{equation}\label{e:Mf-eqn}
  M(f) = \caD / (\caD f + \caD \widetilde \theta(\Omega^{n-3}(\bA^n))).
\end{equation}
This left $\caD_{\bA^n}$-module
represents the functor of Hamiltonian-invariant sections supported on
$X$:
\begin{equation}\label{e:Mf-pr}
  \Hom_{\caD}(M(f), N) = \Gamma_X(\bA^n,N)^{\Omega^{n-3}(\bA^n)},
\end{equation}
where $N$ is any left $\caD$-module
and $\Gamma_X(\bA^n,N) = \{n \in N \mid
fn=0\}$ are the elements of $N$ which are scheme-theoretically
supported on $X$.  


The left $\mathcal{D}$-module $M(f)$ and $\mathcal{D}[s]f^s/\mathcal{D}[s]f^{s+1}$ have striking similarities. They are both supported on $X$ and their
D-module restrictions to 
the smooth locus $X^\circ$ of $X$ are both
isomorphic to the structure sheaf $\mathcal{O}_{X^\circ}$ viewed as a left
$\caD_{X^\circ}$-module
 (for $M(f)$, this
follows from the corresponding statement for $M(X,\theta)$ in \cite[Example 2.37]{ES-dmlv}).
Moreover, one easily deduces from the presentation (\ref{e:Mf-eqn}) that there is a unique morphism of left $\mathcal{D}$-modules $M(f) \to \caD[s] f^s / \caD[s] f^{s+1}$ sending $1$ to $f^s.$

\begin{definition} \label{defi: canonical morphism} The canonical
  morphism $M(f) \xrightarrow{\gamma_f} \caD[s] f^s / \caD[s] f^{s+1}$
  is the unique morphism of left $\mathcal{D}$-modules such that
  $\gamma_f(1)=f^s.$
\end{definition}
Our next main result clarifies the situation when $f$ is quasi-homogeneous with an isolated singularity:

\begin{theorem} \label{t:comp} Let $f$ be a complex polynomial in $n\geq3$ variables. 
If $f$ is quasi-homogeneous with an isolated singularity,
  then $\gamma_f$ is an isomorphism.
\end{theorem} 

The theorem is proved in Section \ref{s:pf-t:comp}.
When $f$ is quasi-homogeneous with an isolated singularity, Theorem
\ref{t:comp} allows us to use the results of \cite{MR3283930} on the
structure of
$M(f)$ 
to study $\caD \frac{1}{f} / R$ and the canonical surjection
$\caD[s] f^s / \caD[s] f^{s+1} \to \caD \frac{1}{f} /
R.$ It also allows us to describe the precise structure of the $\caD$-module
$\caD[s] f^s / \caD[s] f^{s+1}$. Let $\mu_x$ be the Milnor number of the singularity at $x$.
\begin{corollary} \label{c:str-ds}
If $f$ is quasi-homogeneous with an isolated
  singularity at $x$, then there is a decomposition
  $\caD[s]f^s / \caD[s]f^{s+1} \cong \delta^{\mu_x - g_x} \oplus N$, where
  $N$ is an indecomposable $\caD$-module admitting a filtration
  $0 \subseteq N_1 \subseteq N_2 \subseteq N_3= N$ such that
  $N_3/N_2 \cong \delta^{g_x}$, $N_2 / N_1 \cong \IC(X)$, and
  $N_1 \cong \delta^{\dim H^{n-2}(X^\circ,\bC)}$.
\end{corollary}
\begin{proof} This follows from the theorem together with
  \cite[Theorem 2.7, Theorem 3.1, Proposition 3.2]{MR3283930}, using Lemma
  \ref{l:dim-gamma} to obtain the constant $g_x$ and
  \cite[(2.9)]{MR3283930} (or Lemma
  \ref{l:j!-sub} below) to obtain $\dim H^{n-2}(X^\circ,\bC)$.
\end{proof}


\begin{remark}
  It is interesting to ask when the canonical morphism remains
  surjective or injective in the inhomogeneous case.  One simple
  observation is the following: suppose $\lambda$ is a root of the
  $b$-function such that $\lambda-m$ is not a root for any $m \geq
  1$. In this case, it is well-known that the quotient $\caD f^\lambda
  / \caD f^{\lambda +1 }$ is nonzero (\cite[Proposition
  6.2]{MR0430304}).
  Similarly, one can prove that the image of the canonical map
  $M(f) \to \caD[s] f^s / \caD[s] f^{s+1}$ contains the entire
  generalized eigenspace $(\caD[s] f^s / \caD[s] f^{s+1})_{(\lambda)}
  := \{m \in \caD[s] f^s / \caD[s] f^{s+1} \mid (s-\lambda)^N m = 0, N
  \gg 0\}$ under $s$ of eigenvalue $\lambda$. 
\end{remark}
\begin{remark}\label{r:ge-ds}
  Using Corollary \ref{c:str-ds}, we can describe the generalized
  eigenspaces of $\caD[s] f^s / \caD[s] f^{s+1}$ under multiplication
  by $s$.  At eigenvalue $\lambda \neq -1$, the generalized eigenspace
  is just $\delta^{m_\lambda}$ for $m_\lambda$ given in Corollary
  \ref{c:lneq-1}. At eigenvalue $\lambda=-1$, it follows from Theorems
  \ref{t:mt-qh} and \ref{theorem: ker=db} below that the generalized
  eigenspace is the indecomposable D-module $N$ appearing in Corollary
  \ref{c:str-ds}.
\end{remark}
\begin{remark} The definition of $M(f)$ can be extended to the case
  when $\bA^n$ is replaced by a Calabi--Yau variety, i.e., a smooth
  variety with a nonvanishing algebraic volume
  form: we simply replace $\partial_1 \wedge \cdots \wedge \partial_n$ by
  the inverse of the volume form (see \cite[\S 3.4]{ES-dmlv}).
Theorem \ref{t:comp} generalizes directly to that setting. 
  \end{remark}

\subsection{Relation to work of M.~Saito}\label{s:Saito}
After independently obtaining our results, we were informed by
M.~Saito of similar results in \cite{Sai-dmgrphf} obtained by
different methods.  In particular, he describes $\caD f^\lambda / \caD
f^{\lambda+1}$ in terms of the $V$-filtration on $\caD[s] f^s$.
Combining \cite[Corollary 1]{Sai-dmgrphf} with Corollary
\ref{c:lneq-1} above, we recover the following result, which is
essentially \cite[Theorem 1]{Ste-ifqhs} (summing the statement below
over all $\beta$ yields \cite[Theorem 1]{Ste-ifqhs} except without the
statement about the weight filtration).  Let $F_{f,0}$ be the Milnor
fiber of $f$ around $0$.  The Milnor cohomology $H^{n-1}(F_{f,0},\bC)$
has a natural monodromy action of $\pi_1(\bC\setminus\{0\}) \cong
\bZ$, and we let $H^{n-1}(F_{f,0},\bC)_{(\lambda)}$ denote the
generalized eigenspace of eigenvalue $\lambda$ with respect to the
standard generator.  Let $F$ be the Hodge filtration.
\begin{theorem}\label{t:steenbrink} 
Let
$0 < \beta \leq 1$, 
$j \in \bZ$, and
  $\lambda := -j-\beta$. 
If $\lambda \neq -1$, then the vector space
  $\gr_F^{n-j-1} H^{n-1}(F_{f,0},\bC)_{(\exp(-2\pi i \beta))}$ is nonzero if
  and only if 
  $\lambda$ is a root of the $b$-function of
  $f$. Moreover, for all $\lambda$, we have
\begin{equation}
\gr_F^{n-j-1}
  H^{n-1}(F_{f,0},\bC)_{(\exp(-2\pi i \beta))}\cong J_{-|f|\lambda - \sum_i m_i}.
\end{equation}
\end{theorem}
In other words, the Jacobi ring is the associated graded vector space
of the Milnor cohomology with respect to the Hodge filtration (up to
shifting degrees and taking into account the monodromy
eigenvalues).  

One can in fact show that any two of the results: (1) Corollary \ref{c:lneq-1},
(2) Steenbrink's Theorem \ref{t:steenbrink}, and
(3) Saito's \cite[Corollary 1]{Sai-dmgrphf}, imply the third.




\subsection{Notations} \label{subs: notations} For the rest of the paper and unless otherwise mentioned, we will use the following notations: 
$R$ is the ring of complex polynomials in $n \geq 3$ variables and $f$ is a nonconstant element of $R.$ We let $$\{f=0\}=: X \xrightarrow{i} \bA^n \xleftarrow{j} U:= \{f\not=0\}$$ denote the canonical inclusions. Moreover we note $X^\circ \xrightarrow{j^X} X$ the canonical open embedding, where $X^\circ$ is the smooth locus of $X.$ Let $\delta_x$ denote the irreducible D-module supported at a point $x \in X$ and let $(-)^\vee$ denote the dual vector space.  
Finally, for a morphism $S \xrightarrow{\alpha} T$ of varieties, $(\alpha_*, \alpha^*, \alpha_!, \alpha^!, \alpha_{!*})$ denote the usual functors between the bounded derived categories of holonomic D-modules.


\section{Proof of Conjecture \ref{c:mc} in the quasi-homogeneous case} \label{s:pf-c:mc}


The goal of this section is to prove the following result:
\begin{theorem} \label{t:mt-qh} Let $f$ be a quasi-homogeneous complex
  polynomial in $n\geq3$ variables with an isolated singularity.
Then the length of the left $\mathcal{D}$-submodule of
  $H^1_f(R)$ generated by $\frac{1}{f}$ is $1 + g$, with $g$ the reduced genus of the
  singularity.  
\end{theorem}
We will reduce the proof to a statement about $M(f)$, of which much is known thanks to \cite{MR3283930}.
Let us begin with a couple of lemmas. 

\begin{lemma} \label{lemma: appendix} Let $V \xrightarrow{j'} Y$ be an open embedding of not necessarily smooth algebraic varieties and let $Z$ be the complement of $V$ in $Y.$ Let $M$ and $N$ be D-modules on $V$ and $Y$, respectively.  Suppose that there is an isomorphism $j'^*N \xrightarrow{\psi} M$, and let $N \xrightarrow{\phi} j'_*M$ be the adjoint morphism. Then $N \xrightarrow{H^0(\phi)} H^0j'_*M$ is an isomorphism if and only if $\Hom(K, N)= \Ext^1(K, N)= 0$, for all D-modules 
$K$ on $Y$  which are supported on $Z.$ 
\end{lemma}

\begin{proof} For the only if, let $K$ be a D-module supported on $Z$.
  Note that there is a first-quadrant spectral sequence whose $(i,k)$
  entry is $\Ext^i(K, H^kj'_* M)$, with differential of degree
  $(2,-1)$, which converges to $\Ext^{i+k} (K, j'_* M)$. Since on
  subsequent pages the differential is of degree $(2+m,-1-m)$ for
  $m > 0$, we see that $\Ext^i(K,H^0j'_* M)$ survives on all pages for
  $i \leq 1$ and hence to $\Ext^{i}(K,j'_* M)$.  On the other hand,
  applying adjunction,
  $\Ext^{i}(K,j'_*M) \cong \Ext^{i}(j'^* K, M) = 0$.  Therefore
  $\Ext^i(K,H^0j'_* M) = 0$ for $i \leq 1$.  As a result,
  $N \cong H^0 j'_* M$ implies that $\Hom(K,N)=\Ext^1(K,N)=0$.

  Let us prove the if part. Assume therefore that $\Hom(K,N)=\Ext^1(K,N)=0$ for all $K$ supported on $Z$.
Observe that $\ker H^0(\phi)=0.$ Indeed
  $\ker H^0(\phi)$ and $\coker H^0(\phi)$ are supported on $Z$ since
  $\psi$ is an isomorphism but then $\Hom(\ker H^0(\phi), N)=0$, by
  hypothesis on $N.$
We thus have a short exact sequence: 
$$0\to N \xrightarrow{H^0(\phi)} H^0j_*M \to \coker H^0(\phi) \to 0.$$ It has to split since $\Ext^1(\coker H^0(\phi), N)= 0$ by hypothesis. But by adjunction (or the only if part), $\Hom(\coker H^0(\phi), H^0j_*M)= 0$, 
so $\coker H^0(\phi)=0.$
\end{proof}

Recall
  notations from \ref{subs: notations}. 
\begin{lemma} \label{l:j*maxind} 
Suppose that $f$ has an
  isolated singularity at the origin. Then the left
  $\mathcal{D}$-module $H^1_f(R)$ has no submodules supported at the
  origin.  Moreover, we have an isomorphism
  $H^1_f(R) \cong H^0(i \circ j^X)_*
  \cO_{X^\circ}.$ 
\end{lemma}

\begin{proof}
The long exact sequence for $0 \to R \to R[\frac{1}{f}]
\to H^1_f(R) \to 0$ yields:
\[
0\to \Hom(\delta_0, R) \to \Hom(\delta_0, R[\frac{1}{f}]) \to
\Hom(\delta_0, H^1_f(R)) \to \Ext^1(\delta_0,
R) \to \cdots.
\]
Moreover, for all
$\ell, \Ext^\ell(\delta_0, R[\frac{1}{f}]) = \Ext^\ell(j^* \delta_0, \cO_U) =
0$
by adjunction. Hence it follows that
$\Ext^\ell(\delta_0, H^1_f(R)) \cong \Ext^{\ell+1}(\delta_0, R)$ for all
$\ell$.  Since $\Ext^{\ell+1}(\delta_0, R)$ vanishes if $\ell+1\not= n$, and
$n\geq3$ by assumption, we certainly have that
$\Hom(\delta_0, H^1_f(R))=0$, proving the first part of the lemma, as
well as $\Ext^1(\delta_0, H^1_f(R))=0.$ Note finally that $H^1_f(R)$ is
supported on $X$ and its restriction to the smooth locus $X^\circ$ is
isomorphic to $\cO_{X^\circ}.$ Hence we may apply Lemma \ref{lemma:
  appendix} to the D-module on $X$ corresponding to
$H^1_f(R)$ and deduce that 
$H^1_f(R)\cong i_*(H^0j^X_*\cO_{X^\circ}) \cong H^0(i \circ j^X)_*
\cO_{X^\circ}$, as claimed.
\end{proof}

We now precisely relate $\caD \frac{1}{f} / R$ to $M(f).$ Recall the canonical morphism $\gamma_f$ from Definition \ref{defi: canonical morphism} and note that it induces a surjective morphism $M(f) \xrightarrow{\alpha_f} \caD \frac{1}{f} / R.$ 

\begin{proposition} \label{prop: alpha} Suppose that $f$ 
has a unique 
singularity. Then the kernel of $\alpha_f$ is the maximal submodule of $M(f)$ supported at the singularity.
\end{proposition}

\begin{proof} We claim that the left $\mathcal{D}$-module $M(f)$ is
  holonomic. This is a consequence of \cite[Corollary 3.37]{ES-dmlv}
  (and is probably well-known), but let us give a direct proof for the
  reader's convenience. First note that $M(f)$ is supported on $X$ and
  the corresponding D-module on $X$ 
  restricts on $X^\circ$ to the left $\caD_{X^\circ}$-module
  $\mathcal{O}_{X^\circ}$. Therefore the singular support of $M(f)$ is
  contained in the union of the zero section $X \subseteq T^* X$ and
  the cotangent fiber $T^*_xX$ at the singularity $x \in X$, which is
  Lagrangian. This proves the claim.

  Let us consider the maximal submodule $K$ of $M(f)$ supported at the
  singularity. The first assertion of Lemma \ref{l:j*maxind} then
  implies that $\alpha_f(K)=0.$ Let us show that the induced
  surjective morphism
  $M(f)/K \xrightarrow{\bar{\alpha_f}} \caD \frac{1}{f} /R$ is
  injective. We see that $i_*j^X_{!*}\mathcal{O}_{X^{\circ}}$ is the
  unique minimal submodule of $M(f)/K.$ Hence $\ker\bar{\alpha_f}$ is
  either trivial or contains $i_*j^X_{!*}\mathcal{O}_{X^{\circ}}.$ But
  the latter cannot occur since $\caD \frac{1}{f} / R$ is not
  supported at the singularity.
\end{proof}

We may now prove Theorem \ref{t:mt-qh}.

\begin{proof}
  Let the cone point, i.e., the singularity, be the
  origin $0 \in X$. As in \cite{MR3283930},
let $M_{\max}$ denote the (unique) indecomposable factor of
  $M(X, \theta)$ fully supported on $X.$ By Proposition
  \ref{prop: alpha},  the map $M_{\max} \to \caD \frac{1}{f}/R$ is surjective
with kernel the maximum submodule of $M_{\max}$ supported at the origin: call this kernel $K'$.
Thus, the lengths of $\caD \frac{1}{f} / R$ and
  $M_{\max}/K'$ are the same. Next, by \cite[Theorem 3.1]{MR3283930}, 
  $H^0 j_! \Omega_{X^\circ} \subseteq M_{\max}$, where $\Omega_{X^\circ}$ is the right D-module of volume forms. Since $M_{\max}$ is indecomposable,
  $K'$ is also the maximal submodule of $H^0 j_! \Omega_{X^\circ}$ supported at the origin.  Since $H^0 j_! \Omega_{X^\circ}$ has no quotient supported at the origin, $H^0 j_! \Omega_{X^\circ}/K'$ must be the minimal extension $\IC(X)$ of $\Omega_{X^\circ}$, which is simple.  Therefore
the length of
  $M_{\max}/K'$ is $1+ \ell$, 
  with $\ell$
  the length of $M_{\max}/H^0j_!\Omega_{X^\circ}$.
  By \cite[Proposition 3.2]{MR3283930},
$\ell$
  equals $\dim
  \Gamma(X, \mathcal{O}_X)_d$, with $d$ the weight of
  $X$,
  that is the degree of $f$
  in the grading on $R$
  introduced in Subsection \ref{ss:dmrbf}. 

  It remains to show that $\dim
  \Gamma(X, \mathcal{O}_X)_d= g$. This is 
  proved in the following lemma.
\end{proof}
The conclusion relies on the following identity:
\begin{lemma}\label{l:dim-gamma}
 Let $f$ be 
 as in Theorem \ref{t:mt-qh}, then 
$\dim \Gamma(X,\mathcal{O}_X)_d = g$.
\end{lemma}
\begin{proof}
  Let $0 \in X$ be the singularity. Let $\rho: \widetilde X \to
  X$ be a normal crossings resolution of singularities and $Y :=
  \rho^{-1}(0)$. By definition, $g = \dim
  H^{n-2}(Y,\cO_Y)$. It therefore suffices to prove that
  $H^{n-2}(Y,\cO_Y)
  \cong
  \Gamma(X,\mathcal{O}_X)_d$.  For this we follow \cite[\S 3]{MR3283930}
  (and the corrected online version of Proposition 3.13), providing
  details for the reader's convenience.

  Since $Y$ is a normal crossings divisor, its logarithmic canonical
  bundle $\Omega^{n-2}_{Y,\text{log}}$, consisting of volume forms on
  the smooth locus of $Y$ which have only simple poles along the
  intersections of irreducible components of $Y$, 
  is isomorphic to the canonical sheaf.
  By Grothendieck--Serre duality, we obtain that
  $H^{n-2}(Y,\cO_Y) \cong H^0(Y,\Omega^{n-2}_{Y,\text{log}})$.  Now,
  we can form the exact sequence
  \begin{equation} \label{e:es-ciis-eq} 0 \to \Omega^{n-1}_{\widetilde
      X} \to \Omega^{n-1}_{\widetilde X}(Y) \to
    \Omega^{n-2}_{Y,\text{log}} \to 0.
\end{equation}
By Grauert--Riemenschneider vanishing,
$R^1 \rho_* \Omega^{n-1}_{\widetilde{X}} = 0$, and since $X$ is
affine, we obtain that $H^1(X,\rho_* \Omega^{n-1}_{\widetilde{X}}) = 0$.
Therefore $H^1(\widetilde{X}, \Omega^{n-1}_{\widetilde X})=0$. By
\eqref{e:es-ciis-eq}, we obtain that
$\Gamma(Y,\Omega^{n-2}_{Y,\text{log}}) \cong \Gamma(\widetilde{X},
\Omega^{n-1}_{\widetilde X}(Y)) / \Gamma(\widetilde X,
\Omega^{n-1}_{\widetilde X})$.
Finally note that $\Gamma(\widetilde X, \Omega^{n-1}_{\widetilde X})$
and $\Gamma(\widetilde X, \Omega^{n-1}_{\widetilde{X}}(Y))$ are both
subspaces of $\Gamma(X^\circ, \Omega^{n-1}_X)$, global volume forms on
$X^{\circ}$, which is identified with
$\Gamma(X^\circ,\cO_{X})=\Gamma(X,\cO_X)$ under the isomorphism
sending $dx_1 \wedge \cdots \wedge dx_n / df$ to $1$. This isomorphism
identifies $\Gamma(\widetilde X, \Omega^{n-1}_{\widetilde X})$ with
$(\cO_{X})_{> d}$, since a volume form on $X^\circ$ is a sum of
homogeneous forms, and a homogeneous form $\alpha$ extends to
$\widetilde X$ if and only if it is of positive degree, i.e., in some
neighborhood $U$ of $0 \in X$, the integral
$\lim_{\varepsilon \to 0} \int_{U \cap \{|f| \geq \varepsilon\}}
\alpha \wedge \overline{\alpha}$
converges. Similarly,
$\Gamma(\widetilde X, \Omega^{n-1}_{\widetilde X}(Y))$ identifies with
$(\cO_X)_{\geq d}$, since a meromorphic volume form $\alpha$ on
$\widetilde X$ has logarithmic poles at $Y$ if and only if it has
nonnegative degree, i.e., the limit
$|\log \varepsilon|^{-1} \int_{U \cap \{|f| > \varepsilon\}} \alpha
\wedge \overline{\alpha}$
exists for some neighborhood $U$ of $0 \in X$.  Put together we obtain
an isomorphism
$\Gamma(Y, \Omega^{n-2}_{Y,\text{log}}) \cong (\cO_{X})_d$ as
desired.
\end{proof} 

\begin{remark} The proof implies that Conjecture \ref{c:mc} is equivalent to \cite[Conjecture 3.8]{MR3283930}.
\end{remark}

\section{Proof of Theorem \ref{t:comp}}\label{s:pf-t:comp}
We want to show that if $f$ is quasi-homogeneous with an isolated singularity, then the canonical morphism $M(f) \xrightarrow{\gamma_f} \caD[s] f^s / \caD[s] f^{s+1}$ is an isomorphism. Since $f$ is quasi-homogeneous, there is a vector field $v$ such that $v(f^s)=sf^s.$ Thus $\caD[s] f^s / \caD[s] f^{s+1}$ is generated by $f^s$ as a $\caD$-module and hence $\gamma_f$ is surjective. We will prove that $\gamma_f$ is also injective by showing that the lengths of $M(f)$ and $\caD[s] f^s / \caD[s] f^{s+1}$ are the same.

\subsection{Nearby cycles}

We will use our running notations and hypotheses, see Subsection
\ref{subs: notations}. We set $\caD_U:=\caD[\frac{1}{f}]$ to be the
ring of differential operators on $U.$ Let $M$ be a holonomic left
$\caD_U$-module $M.$ We denote the nearby cycles $\caD$-module with
respect to $f$ by $\Psi_f(M)$ \cite{MR726425, MR737934}. This is a
holonomic left $\caD$-module supported on $X$ and is equipped with a
log-monodromy operator $s$, such that the classical monodromy operator
$T$ (in the case that $M$ has regular singularities) is given by
$T=e^{2\pi is}$.

First and for later use, we would like to relate nearby cycles to the
quotients $\caD f^\lambda / \caD f^{\lambda+1}.$
\begin{definition} Let $\lambda\in \mathbb{C}.$ We
  define: \begin{enumerate}
  \item $\caD f^{\lambda-\infty}$ is the left $\caD$-module
    $R[\frac{1}{f}] f^\lambda$
  \item $F$ is the filtration of $\caD f^{\lambda-\infty}$ by
    $\caD$-submodules given by
    $F^i \caD f^{\lambda-\infty} := \caD f^{\lambda+i}$
  \item $\caD f^{\lambda+\infty}$ is the intersection
    $\bigcap_{i \in \bZ} F^i \caD f^{\lambda-\infty}$
\end{enumerate}
\end{definition}

\begin{remark} \label{remark: associated graded} Note that the
  associated graded $\caD$-module $\gr^F \caD f^{\lambda-\infty}$ is
  the sum
  $\caD f^{\lambda+\infty} \oplus \bigoplus_{i \in \bZ} \caD
  f^{\lambda+i} / \caD f^{\lambda+i+1}.$
  Clearly the filtration $F$, and hence the sum, are actually finite,
  since $\caD f^{\lambda-\infty}$ is a holonomic $\caD$-module.
\end{remark}

\begin{remark}
$F$ seems to be closely related to the filtration $G$ used by M. Saito in  \cite[\S 1.3]{Sai-dmgrphf}. 
\end{remark}

We will need 
a precise relationship between nearby cycles and $\caD f^{\lambda-\infty}$.

\begin{proposition} \label{prop: Df nearby cycles}
Let $\lambda\in\mathbb{C}$ and let $\cO^\lambda_U$ denote the left $\caD_U$-module
$R[\frac{1}{f}] f^\lambda.$ Then
\begin{enumerate}
\item $\caD f^{\lambda+\infty} \cong j_{!*} \cO^\lambda_U$
\item  $\caD f^{\lambda-\infty}/\caD f^{\lambda+\infty} \cong
  \coker(\Psi_f(\cO_U) \xrightarrow{T-e^{2\pi i\lambda}} \Psi_f(\cO_U))$
  \end{enumerate}
\end{proposition}

\begin{proof} This is well-known. The first assertion is \cite[lemma
  3.8.2]{MR833194}. The second assertion follows by the Riemann-Hilbert correspondence from its perverse sheaf counterpart. The latter is an immediate consequence of \cite[Section 5.6 p.617]{MR2525735} since $\coker(\Psi_f(\cO_U) \xrightarrow{T-e^{2\pi i\lambda}} \Psi_f(\cO_U)) \cong \coker(\Psi_f(\cO_U^\lambda) \overset{T- \Id}{\to} \Psi_f(\cO_U^\lambda))$.
  \end{proof}

We now come to the computation of the length of nearby cycles.
We will need the following lemma, which is a generalization of
\cite[(2.9)]{MR3283930}.

\begin{lemma} \label{l:j!-sub} Suppose $X$ has an isolated singularity
  at $x \in X$ and let $V$ be a contractible analytic neighborhood of
  $x$, such that $V\setminus\{x\}$ is smooth, with
  $j^V := j^X|_{V\setminus\{x\}}: V\setminus\{x\}\to V$ the inclusion.
  Then, the kernel of the canonical surjection
  $H^0j_!^V\cO_{V\setminus\{x\}} \to j_{!*}\cO_{V\setminus\{x\}}$ is isomorphic to
  $\delta_x \otimes H^{n-2}(V\setminus\{x\}, \mathbb{C})^\vee.$
  Similarly, the cokernel of the canonical injection
  $j_{!*} \cO_{V\setminus\{x\}} \to H^0 j_*^V \cO_{V\setminus\{x\}}$ is isomorphic to
  $\delta_x \otimes H^{n-2}(V\setminus\{x\}, \mathbb{C})$.
\end{lemma}
The argument closely follows \cite[Lemma 4.3]{ESdm}, but as the
statement is more general we provide the proof. Note that in this
statement and proof it is not necessary that $X$ be a hypersurface.
\begin{proof} Consider the exact sequence
  $ 0 \to K \to H^0j_!^V\cO_{V\setminus\{x\}} \to
  j_{!*}\cO_{V\setminus\{x\}} \to 0$
  with $K$ supported at $x$.  Apply $\Hom(\delta_x, -)$ to obtain the
  isomorphism
  $$\Hom(\delta_x, K) \cong \Hom(\delta_x,
  H^0j_!^V\cO_{V\setminus\{x\}}).$$
  Let $i^x: \{x\} \to V$ be the closed inclusion.  By the adjunction
  $(i_!, H^0 i^!)$,
  $$\Hom(\delta_x,H^0j_!^V \cO_{V\setminus\{x\}})=H^0 i^! H^0 j_!^V
  \cO_{V\setminus\{x\}}.$$
  Now, the cohomology of $j_!^V \cO_{V\setminus\{x\}}$ is concentrated
  in nonpositive degrees (since $j_!^V$ is the left derived functor
  of the left adjoint $H^0 j_!^V$ of the exact restriction functor
  $(j^V)^!$). Moreover, $H^{<0} j_!^V \cO_{V\setminus\{x\}}$ is
  concentrated at the singularity $x$. Therefore
  $H^0 i^! j_!^V \cO_{V\setminus\{x\}} = H^0 i^! H^0 j_!^V
  \cO_{V\setminus\{x\}}$.
  The former is Verdier dual to $H^0 i^* j_*^V \cO_{V\setminus\{x\}}$,
  which is nothing but the zeroth cohomology of the stalk at $x$ of
  $j_*^V \cO_{V\setminus\{x\}}$, i.e., the cohomology
  $H^{\dim V}(V\setminus\{x\},\bC) = H^{n-2}(V\setminus\{x\}, \bC)$. This proves the
  first assertion.  The second assertion follows from the first by
  Verdier duality.
\end{proof}



\begin{proposition} \label{p:len-ns}
Suppose that $f$ is quasi-homogeneous with an isolated singularity. Then the length of the nearby cycles $\caD$-module $\Psi_f(\cO_U)$ is $1+ \mu + \dim H^{n-2}(X^\circ, \mathbb{C})$, with $\mu$ the Milnor number of the singularity.
\end{proposition}

\begin{proof} 
  By Proposition \ref{prop: Df nearby cycles}, we have
  $H_f^1(R)\cong\caD f^{-1-\infty}/\caD f^{-1+\infty} \cong
  \coker(\Psi_f(\cO_U) \xrightarrow{T- Id} \Psi_f(\cO_U)).$
  Moreover, by Lemma \ref{l:j*maxind},
  $H^1_f(R) \cong H^0(i \circ j^X)_* \cO_{X^\circ}.$ Hence there is a
  short exact sequence
  $0 \to K \to \Psi_f(\cO_U) \to H^0(i \circ j^X)_* \cO_{X^\circ} \to
  0$,
  with $K$ supported at the singularity. Applying Verdier duality, we
  get a short exact sequence
  $0 \to i_*H^0j^X_! \cO_{X^\circ} \to \Psi_f(\cO_U) \to K \to 0.$ By
  duality, it follows from Lemma \ref{l:j*maxind} that
  $i_*H^0j^X_! \cO_{X^\circ}$ has no quotient supported at the
  singularity. Hence $K$ is isomorphic to the maximal quotient
  ${i_x}_* H^0 {i_x}^* \Psi_f(\cO_U)$ of $\Psi_f(\cO_U)$ supported at
  the singularity $x \xhookrightarrow{i_x} \bA^n.$ But
  $H^0 {i_x}^* \Psi_f(\cO_U)$ is isomorphic to the top reduced
  cohomology of the Milnor fiber of $x$ and hence has dimension $\mu.$
  We thus have that the length $lg(\Psi_f(\cO_U))$ of $\Psi_f(\cO_U)$
  is $lg(i_*H^0j^X_! \cO_{X^\circ}) + \mu.$ The proposition now
  follows directly from Lemma \ref{l:j!-sub}.
\end{proof}

\begin{corollary} \label{cor: length}
Suppose that $f$ is quasi-homogeneous with an isolated singularity. Then the $\caD$-module length of $\caD[s] f^s / \caD[s] f^{s+1}$ is $1+ \mu + \dim H^{n-2}(X^\circ, \mathbb{C})$, with $\mu$ the Milnor number of the singularity.
\end{corollary}

\begin{proof} By \cite[Proposition 5.6]{MR833194}, $\caD[s] f^s / \caD[s] f^{s+1}$ and $\Psi_f(\cO_U)$ have the same length.
Hence the statement is equivalent to Proposition \ref{p:len-ns}.
\end{proof}

\subsection{Length of $M(f)$}

\begin{proposition} \label{prop: length M(f)}
Suppose that $f$ is quasi-homogeneous with an isolated singularity. Then the $\caD$-module length of $M(f)$ is $1+ \mu + \dim H^{n-2}(X^\circ, \mathbb{C})$, with $\mu$ the Milnor number of the singularity.
\end{proposition}

\begin{proof} By \cite[Theorem 2.7]{MR3283930}, there is a short exact sequence $0 \to i_*H^0 j^X_! \cO_{X^\circ} \to M(f) \to \delta^{\mu} \to 0.$ Hence the proposition follows from Lemma \ref{l:j!-sub}. 
\end{proof}

The theorem easily follows:

\begin{proof}[Proof of Theorem \ref{t:comp}]
The canonical morphism $M(f) \xrightarrow{\gamma_f} \caD[s] f^s / \caD[s] f^{s+1}$ is surjective since $f$ is quasi-homogeneous, as explained at the beginning of the section. Since $M(f)$ has the same length as 
$\caD[s] f^s / \caD[s] f^{s+1}$ 
by Proposition \ref{prop: length M(f)} and Corollary \ref{cor: length}, $\gamma_f$ is also injective.
\end{proof}

\section{Proof of Theorem \ref{t:lneq-1} and the quotients $\caD f^\lambda / \caD f^{\lambda+1}$}\label{s:pf-t:lneq-1}

The following notion is central to our proof. 

\begin{definition} Let $Z \xhookrightarrow{i_Z} \bA^n$ be a closed embedding and let $M$ be a holonomic left $\caD$-module.
Then the holonomic $\caD$-module $\db_Z(M)$ supported on $Z$ is ${i_Z}_* \ker( H^0 {i_Z}^! M \to H^0 {i_Z}^* M).$ \end{definition}
Here ``db'' stands for ``deltas on the bottom.''


\begin{remark} \label{rmk: db}
It follows directly from the definition that $\db_Z(M)$ is the kernel of the canonical map from the maximal submodule of $M$ supported at $Z$ to the maximal quotient of $M$ supported at $Z.$
\end{remark}

We need the following easy lemma on $\db_Z:$

\begin{lemma} \label{lemma: db}
\begin{enumerate}
\item \label{item: db image}
Let $M\xrightarrow{\alpha}N$ be a morphism of holonomic $\caD$-modules. Then $\alpha(\db_Z(M))\subseteq \db_Z(N).$

\item \label{item: db sums}
The formation of $\db_Z$ commutes with finite direct sums.

\end{enumerate}
\end{lemma}

\begin{proof} It follows easily from the description of $db_Z$ in Remark \ref{rmk: db}.
\end{proof}

We will deduce Theorem \ref{t:lneq-1} from the following. We let $(\caD[s] f^s / \caD[s] f^{s+1})_{(\lambda)}$ be the generalized eigenspace under $s$ of eigenvalue $\lambda.$

\begin{theorem} \label{theorem: ker=db}
Suppose that $f$ is quasi-homogeneous with an isolated singularity at $0.$
Let $p$ be the surjective morphism $\caD[s] f^s / \caD[s] f^{s+1} \xrightarrow{\Sigma_\lambda p_\lambda} \bigoplus_\lambda \caD f^\lambda / \caD f^{\lambda+1}$, 
where $p_\lambda$ is the natural surjection $(\mathcal{D}[s]f^s/\mathcal{D}[s]f^{s+1})_{(\lambda)} \to
\mathcal{D}f^{\lambda}/\mathcal{D}f^{\lambda+1}.$ 
Then $\ker p= \ker p_{-1}$ 
is the submodule $\db_{\{0\}}(\caD[s] f^s / \caD[s] f^{s+1})$, and the latter is isomorphic to $\delta^{\dim H^{n-2}(X^\circ, \mathbb{C})}.$
\end{theorem}

\begin{proof} Let us first prove that
  $\db_{\{0\}}(\caD[s] f^s / \caD[s] f^{s+1})$ is isomorphic to
  $\delta^{\dim H^{n-2}(X^\circ, \mathbb{C})}.$ By \cite[Theorem
  2.7]{MR3283930}, there is a short exact sequence
  $0 \to i_*H^0 j^X_! \cO_{X^\circ} \to M(f) \to \delta^{\mu} \to 0$
  and the map $M(f) \to \delta^{\mu}$
is the maximal quotient of $M(f)$ supported at the
  origin. Hence $\db_{\{0\}}(M(f))$ is the maximal submodule of
  $i_*H^0 j^X_! \cO_{X^\circ}$ supported at the origin. By Lemma
  \ref{l:j!-sub}, the latter is isomorphic to
  $\delta^{\dim H^{n-2}(X^\circ, \mathbb{C})}.$ The claim thus follows
  from Theorem \ref{t:comp}.

  We now prove that
  $\db_{\{0\}}(\caD[s] f^s / \caD[s] f^{s+1})\subseteq \ker p.$ We
  have
  $\db_{\{0\}}(\bigoplus_\lambda \caD f^\lambda / \caD
  f^{\lambda+1})=0.$
  Indeed, by (\ref{item: db sums}) of Lemma \ref{lemma: db}, it
  suffices to check that
  $\db_{\{0\}}(\caD f^\lambda / \caD f^{\lambda+1})$ vanishes for all
  $\lambda.$ For $\lambda\not=-1, \caD f^\lambda / \caD f^{\lambda+1}$
  is supported at the origin hence
  $\db_{\{0\}}(\caD f^\lambda / \caD f^{\lambda+1})=0.$ Moreover by
  Lemma \ref{l:j*maxind}, $\caD f^{-1}/R \subseteq H^1_f(R)$ does not
  have any submodule supported at the origin. Thus
  $\db_{\{0\}}(\caD f^{-1}/R)=0.$ We conclude by (\ref{item: db
    image}) of Lemma \ref{lemma: db}.

  Let us then show that the length of $\ker p$ is at most
  $\dim H^{n-2}(X^\circ, \mathbb{C}).$ This immediately implies that
  $\ker p$ is isomorphic to
  $\db_{\{0\}}(\caD[s] f^s / \caD[s] f^{s+1}).$ We claim that the
  length of the module $\bigoplus_\lambda \caD f^\lambda / \caD f^{\lambda+1}$ is
  at least $1 + \mu.$ Since by Corollary \ref{cor: length}, the length
  of $\caD[s] f^s / \caD[s] f^{s+1}$ is
  $1+ \mu + \dim H^{n-2}(X^\circ, \mathbb{C})$, the claim indeed
  implies the upper-bound on the length of $\ker p.$ By Remark
  \ref{remark: associated graded} and Proposition \ref{prop: Df nearby
    cycles}, $\bigoplus_\lambda \caD f^\lambda / \caD f^{\lambda+1}$
  is the associated graded module to a finite filtration on
  $M:= \bigoplus_{\lambda \in \mathbb{C}/\mathbb{Z}}
  \coker(\Psi_f(\cO_U) \xrightarrow{T-e^{2\pi i\lambda}}
  \Psi_f(\cO_U)).$
  Hence their lengths are the same. Let
  $\{0\} \xhookrightarrow{i_0} \mathbb{A}^n$ be the embedding of the
  origin. Since ${i_0}_*H^0i_0^*$ is right-exact, we have that the
  maximal quotient ${i_0}_*H^0i_0^*M$ of $M$ supported at the origin
  is
  $\bigoplus_{\lambda \in \mathbb{C}/\mathbb{Z}}
  \coker({i_0}_*H^0i_0^*\Psi_f(\cO_U) \xrightarrow{T-e^{2\pi
      i\lambda}} {i_0}_*H^0i_0^*\Psi_f(\cO_U)).$
  But the stalk $H^0i_0^*\Psi_f(\cO_U)$ of the nearby cycles at the
  origin is the top reduced cohomology of the Milnor fiber of the
  origin, on which the induced action of $T$ is the usual
  monodromy. Since the monodromy action is semisimple, see for example
  \cite[Chapter 3, Example 1.19]{Dimca}, it follows that the natural
  surjection ${i_0}_*H^0i_0^*\Psi_f(\cO_U) \to {i_0}_*H^0i_0^*M$ is an
  isomorphism. Hence the maximal quotient ${i_0}_*H^0i_0^*M$ of $M$
  supported at the origin is of length $\mu.$ But $M$ contains
  $H^1_f(R)$ and hence is not supported at the origin. Thus its length
  is at least $1+\mu$, as claimed.

Finally, we have that $\db_{\{0\}}(\caD[s] f^s / \caD[s] f^{s+1}) \subset (\caD[s] f^s / \caD[s] f^{s+1})_{(-1)}$, since the other generalized eigenspaces of $\caD[s] f^s / \caD[s] f^{s+1}$ under $s$ are supported at the origin. Hence $$\db_{\{0\}}(\caD[s] f^s / \caD[s] f^{s+1})= \ker p= \ker p_{-1}.$$ This concludes the proof of the theorem.
\end{proof}

\begin{proof}[Proof of Theorem \ref{t:lneq-1}]
By Theorem \ref{theorem: ker=db}, $\ker p= \ker p_{-1}
.$ Hence for all $\lambda\not=-1, \ker q_{\lambda}=0.$ Point (\ref{item: lambda}) immediately follows.

We now prove point (\ref{item: -1}) of the theorem. By definition of $\caD[s] f^s / \caD[s] f^{s+1}|_{s=-1}$, we have a short exact sequence:
$$0\to ((s+1) \caD[s] f^s / \caD[s] f^{s+1})_{(-1)}\to  (\caD[s] f^s / \caD[s] f^{s+1})_{(-1)} \xrightarrow{\pi}\caD[s] f^s / \caD[s] f^{s+1}|_{s=-1} \to0$$ We have $p_{-1}= q_{-1} \circ \pi.$ Hence the kernel $\ker q_{-1}$ is the image $\pi(\ker p_{-1}).$ But by \cite[Theorem 6.18]{MR1943036},
$s$ acts semisimply on $(s+1) \caD[s] f^s / \caD[s] f^{s+1}$ and
$((s+1)\caD[s] f^s / \caD[s] f^{s+1})|_{s=-1} = ((s+1) \caD[s] f^s /
\caD[s] f^{s+1})_{(-1)}= \ker\pi$ is isomorphic to
$\delta^{g_0}.$ 
The result now follows since $\ker p_{-1}$ is isomorphic to
$\delta^{\dim H^{n-2}(X^\circ, \mathbb{C})}$ by Theorem \ref{theorem:
  ker=db}.
\end{proof}

\begin{acknowledgements}
We are grateful to Pavel Etingof for initially 
suggesting a relationship between
$M(X,\theta)$  and $\caD \frac{1}{f} / \mathcal{O}$.
We would like to thank Uli Walther for useful discussions,
particularly for suggesting that Question \ref{q:nontriv} should have an
affirmative answer in the quasi-homogeneous isolated singularity case,
and Johannes Nicaise for pointing out the semisimplicity of monodromy
and the reference \cite{Dimca}, as well as for many useful discussions.
\end{acknowledgements}

\bibliography{master}
\bibliographystyle{amsalpha}
\end{document}